\documentclass{article}
\usepackage[utf8]{inputenc}
\usepackage{graphicx} % Required for inserting images
\usepackage[page]{appendix}

\usepackage{amsmath}
\usepackage{amssymb}
\usepackage{amsfonts}

\usepackage[T1]{fontenc}   % Use T1 encoding for better font rendering
\usepackage{lmodern}       % Load the Latin Modern font family
\usepackage{textcomp}      % Optional: enhances text symbol support

\usepackage{amsthm}
\usepackage{mathtools}
\usepackage{enumerate}
\usepackage[hidelinks]{hyperref}
\usepackage{titlesec}
\usepackage{tikz}

\usepackage[shortlabels]{enumitem}
\usepackage{yfonts}
\usepackage{setspace}
\usepackage[utf8]{inputenc}
\usepackage[english]{babel}
\usepackage[margin = 1in]{geometry}
\usepackage{tikz}
\usepackage{graphicx}
\usepackage{authblk}
\usetikzlibrary{calc}
\usepackage{fancyhdr}
\usepackage{float}

\titleformat{\section}[block]{\centering\Large}{\thesection}{1em}{}

\theoremstyle{definition}

\theoremstyle{remark}

\theoremstyle{plain}

\newtheorem{que}{Question}[section]
\newtheorem{thm}{Theorem}[section]
\newtheorem{prop}[thm]{Proposition}
\newtheorem{lem}[thm]{Lemma}
\newtheorem{cor}[thm]{Corollary}

\newtheorem{defi}[thm]{Definition}

\title{Periodic Inscription of Isosceles Trapezoids}
\author{Ali Naseri Sadr}

\date{}

\begin{document}

\maketitle

\begin{abstract}
    We prove that a pair of continuous disjoint periodic curves in $\mathbb{C}$ inscribes an isosceles trapezoid with any similarity type. The case of smooth curves can be identified with a Lagrangian intersection problem for a pair of Lagrangian cylinders in $\mathbb{R}\times S^1\times\mathbb{C}$, and the continuous case follows from the smooth one by a standard convergence argument.
\end{abstract}

\section{Introduction}
Let $\gamma_1, \gamma_2\colon \mathbb{R}\to \mathbb{C}$ be two continuous embeddings of the real line into $\mathbb{C}$ that satisfy the periodicity condition
\begin{equation*}
    \gamma_i(t+1) = \gamma_i(t) + \sqrt{-1}
\end{equation*}
for every $t$ and $i=1, 2$. Furthermore, assume the images of $\gamma_1$ and $\gamma_2$ are disjoint. Tao conjectured in \cite{1}
that there exist four points in $\gamma_1(\mathbb{R})\cup\gamma_2(\mathbb{R})$ which are vertices of a square; this is a variation of the Toeplitz square peg problem for periodic curves, and was proved by Hugelmeyer in \cite{2}.

An isosceles trapezoid is a trapezoid where the two non-parallel sides have equal length; see Figure~\ref{Balanced Insc}. 
For any given isosceles trapezoid $Q$, we show there are four points in $\gamma_1(\mathbb{R})\cup\gamma_2(\mathbb{R})$ that are vertices of a quadrilateral similar to $Q$. The approach of \cite{2} does not directly generalize even to the case of rectangles, since the problem is identified to a Lagrangian intersection problem for two compact Lagrangians, and this is only possible if one works with squares. By contrast, in this article, we use a different approach to work with non-compact Lagrangians and prove not only that every pair of periodic curves inscribes every similarity type of rectangles, but also every similarity type of isosceles trapezoids.
  
\begin{defi}
\label{defi1}
     Assume $Q$ is an isosceles trapezoid. We say that the pair $(\gamma_1, \gamma_2)$ admits a \textit{balanced} inscription of $Q$ if there exist
    $p_1,p_2\in\gamma_1(\mathbb{R})$ and $p_3,p_4\in\gamma_2(\mathbb{R})$ such that the quadrilateral formed by $p_1, p_2, p_3, p_4$ is similar to $Q$, the line segments $\overline{p_1p_2}$ and $\overline{p_3p_4}$ are parallel, and $|\overline{p_1p_2}| \leq |\overline{p_3p_4}|$; see Figure~\ref{Balanced Insc}.
\end{defi}
\begin{figure}[h]
\centering
\begin{tikzpicture}[scale=1,rotate=90]
    % Extended Blue curve (originally through p1 and p2)
    % Extended from (-2,2) to (2,2) while still passing through (-1,2) and (1,2)
    \draw[blue, thick] (-2,2)
       .. controls (-1.5,2.5) .. (-1,2)   % p_1
       .. controls (0,2.5) .. (1,2)         % p_2
       .. controls (1.5,2.5) .. (2,2) 
       node [pos=0.5, sloped, above left] {$\gamma_1$};
        
    % Extended Green curve (originally through p3 and p4)
    % Extended from (-2.5,0) to (2.5,0) while still passing through (-1.5,0) and (1.5,0)
    \draw[green, thick] (-2.5,0) 
       .. controls (-2,-0.5) .. (-1.5,0)  % p_3
       .. controls (0,-0.5) .. (1.5,0)      % p_4
       .. controls (2,-0.5) .. (2.5,0)
       node [pos=0.5, sloped, above right] {$\gamma_2$};
    
    % Mark points on the curves (trapezoid vertices remain fixed)
    \filldraw [red] (-1,2) circle (2pt) node[above right] {$p_2$};
    \filldraw [red] (1,2) circle (2pt) node[below right] {$p_1$};
    \filldraw [red] (-1.5,0) circle (2pt) node[above left] {$p_3$};
    \filldraw [red] (1.5,0) circle (2pt) node[below left] {$p_4$};
    
    % Draw the trapezoid connecting the fixed vertices:
    \draw[black] (-1,2) -- (1,2) -- (1.5,0) -- (-1.5,0) -- cycle;
\end{tikzpicture}
\caption{A balanced inscription in $(\gamma_1, \gamma_2)$}
\label{Balanced Insc}
\end{figure}    

Note that our definition depends on the order of the pair $(\gamma_1$, $\gamma_2)$ unless $Q$ is a rectangle.

\begin{thm}
\label{Main Thm}
Suppose $\gamma_1$ and $\gamma_2$ are two continuous disjoint periodic embeddings of the real line into the plane, and suppose $Q$ is an isosceles trapezoid. Then $(\gamma_1$, $\gamma_2)$ admits a balanced inscription of $Q$. Furthermore, there is a generic subset of smooth disjoint periodic pairs such that each pair in this set admits at least two balanced inscriptions of $Q$ that are not related under translation by $\sqrt{-1}$.   
\end{thm}
\begin{cor}
\label{rect cor}
Let $\theta\in (0,\pi)$; then every pair of 
continuous disjoint periodic curves in the plane inscribes a rectangle with angle $\theta$ between its two diagonals.
\end{cor}
The simplest pair of disjoint periodic curves are two parallel vertical lines, and these lines admit a balanced inscription of any trapezoid in the sense of Definition \ref{defi1}. It would be interesting to investigate the following generalization of our theorem.  

\begin{que}
    Does any pair of disjoint periodic curves admit a balanced inscription of every trapezoid?
\end{que}

In contrast to peg problems for closed curves, we can deduce the periodic peg problem for continuous curves from the case of smooth curves using a standard convergence argument, so it suffices to prove the result for a pair of smooth periodic curves.
We prove Theorem \ref{Main Thm} for smooth curves using symplectic geometry. In particular, we will use Lagrangian Floer homology for a pair of non-compact Lagrangian cylinders in $(\mathbb{R}\times S^1\times\mathbb{C}, \omega)$, where the symplectic form depends on the similarity type of $Q$. Our approach draws inspiration from the ideas in \cite{8,3}.

\section*{Acknowledgments}
The author is grateful to his advisors, John Baldwin and Josh Greene, for their invaluable guidance, support, and insightful conversations about this work.

\section{Symplectic Setting}
Consider two disjoint smooth periodic curves $\gamma_1\colon\mathbb{R}\to\mathbb{C}$ and $\gamma_2\colon\mathbb{R}\to\mathbb{C}$ in the plane; we denote $\gamma_i(\mathbb{R})$ by $\gamma_i$ in the following. Fix an isosceles trapezoid $Q$. The similarity type of $Q$ is determined by two pieces of information, the angle between its diagonals and the ratio its two diagonals intersect each other. Let $\theta \in (0,\pi)$ be the angle, and assume the diagonals intersect each other with ratio $\frac{c}{1-c}$ for some fixed $c$ in $(0,\frac{1}{2}]$; see  Figure~\ref{fig1}.

\begin{figure}[h]
    \begin{center}
\begin{tikzpicture}[scale=1]
  % Define the vertices of the trapezoid (taller version)
  \coordinate (A) at (0,0);    % Bottom left (p_3)
  \coordinate (B) at (4,0);    % Bottom right (p_4)
  \coordinate (C) at (3,3);    % Top right (p_1)
  \coordinate (D) at (1,3);    % Top left (p_2)
  
  % Label vertices in counter-clockwise order starting at top right:
  \node[above right] at (C) {$p_1$};
  \node[above left]  at (D) {$p_2$};
  \node[below left]  at (A) {$p_3$};
  \node[below right] at (B) {$p_4$};
  
  % Calculate the intersection of the diagonals (O)
  % Diagonals: A--C and B--D intersect at O:
  % A=(0,0), C=(3,3) so O lies on the line y=x.
  % B=(4,0), D=(1,3) lie on the line y=4-x.
  % Solving x = 4 - x gives x=2 and y=2.
  \coordinate (O) at (2,2);
  
  % Draw the trapezoid
  \draw[thick, blue] (A) -- (B) -- (C) -- (D) -- cycle;
  
  % Draw the diagonals (dashed red)
  \draw[dashed, red] (A) -- (C);
  \draw[dashed, red] (B) -- (D);
  
  % Label segments on diagonal AC (optional)
  \node[below right] at ($ (A)!0.4!(O) $) {$1-c$};
  \node[below right] at ($ (O)!0.5!(C) $) {$c$};
  
  % Draw an arc on the complementary angle, between the rays from O to C and O to D.
  % Compute directions:
  %   O -> C: from O=(2,2) to C=(3,3) gives vector (1,1) with angle 45°
  %   O -> D: from O=(2,2) to D=(1,3) gives vector (-1,1) with angle 135°
  % Draw the arc from 45° to 135° with a radius of 0.5.
  \draw (O) ++(45:0.5) arc (45:135:0.5);
  
  % Place the angle label along the bisector (at 90°) of the arc, offset outward.
  \node at ($(O)+(90:0.7)$) {$\theta$};
  
\end{tikzpicture}
    \end{center}
    \caption{An annotated isosceles trapezoid $Q$}
    \label{fig1}
\end{figure}

We always assume the pair $(c, \theta)$ is fixed in what follows.
Define a map $\psi_c\colon \mathbb{C}^2\to \mathbb{C}^2$ by
\begin{equation*}
    (z,w) \longmapsto (z+cw, z+ (c-1)w).
\end{equation*}
Let $R_\theta\colon\mathbb{C}^2\to\mathbb{C}^2$ denote the map
\begin{equation*}
    (z,w) \longmapsto (z, e^{\sqrt{-1}\theta}w).
\end{equation*}
Observe that four distinct points $p_1, p_2, p_3, p_4$
cyclically label the vertices of an isosceles trapezoid similar to Q, with $\overline{p_1p_2}$ parallel to $\overline{p_3p_4}$
and $|\overline{p_1p_2}| \leq |\overline{p_3p_4}|$, iff there exists $(z, w)\in \mathbb{C}^2$
such that $\psi_c(z, w) = (p_2, p_4)$ and
$\psi_c \circ R_{-\theta}(z, w)=(p_1, p_3)$.
Therefore, the balanced inscriptions of $Q$ in $(\gamma_1, \gamma_2)$ are in one-to-one correspondence with the set 
\begin{equation*}
    \mathcal{I}(\gamma_1,\gamma_2)\coloneqq\psi_c^{-1}(\gamma_1\times\gamma_2)\cap R_\theta(\psi_c^{-1}(\gamma_1\times\gamma_2))
\end{equation*}
where $\gamma_1\times\gamma_2$ denotes the image of $\gamma_1\times\gamma_2:\mathbb{{R}}\times\mathbb{R}\to \mathbb{C}^2$.
We define a pair of symplectic forms on $\mathbb{C}^2$ by 
\begin{align*}
    &\omega_Q\coloneqq dx_1\wedge dy_1 + c(1-c)dx_2\wedge dy_2, \\
    & \omega'_Q\coloneqq (1-c)dx_1\wedge dy_1 + cdx_2\wedge dy_2.
\end{align*}
Define a Hamiltonian $F$ on $(\mathbb{C}^2, \omega_Q)$ by 
\begin{equation}
\label{ham defi}
    F(z,w)\coloneqq \frac{c(1-c)\Vert w\Vert^2}{2}.
\end{equation}
 
Note that $R_\theta$ is the Hamiltonian flow associated to $F$ at time $\theta$ with respect to $\omega_Q$. 
A quick computation shows that $\psi_c^*(\omega'_Q)=\omega_Q$.
The map $\psi_c$ is a diffeomorphism, and since $\gamma_1\times\gamma_2$ is a product Lagrangian in $(\mathbb{C}^2, \omega'_Q)$, we conclude that both $\psi_c^{-1}(\gamma_1\times\gamma_2)$ and its image under $R_\theta$ are Lagrangian submanifolds in $(\mathbb{C}^2, \omega_Q)$. We consider two actions of $\mathbb{Z}$ on $\mathbb{C}^2$, one generated by $T_1(z,w) = (z+\sqrt{-1}, w+\sqrt{-1})$ and one generated by $T_2(z,w)=(z+\sqrt{-1}, w)$. We have
\begin{equation*}
    \psi_c\circ T_2 = T_1\circ\psi_c.
\end{equation*}
Notice that $T_2$ is a symplectomorphism of $(\mathbb{C}^2, \omega_Q)$, and $\gamma_1\times\gamma_2$ is invariant under $T_1$ because we assumed both curves are periodic. Hence, $\psi_c^{-1}(\gamma_1\times\gamma_2)$ is invariant under $T_2$. Additionally, we have $F\circ T_2=F$; thus the Hamiltonian flow associated to $F$ is invariant under $T_2$, and this shows $R_\theta(\psi_c^{-1}(\gamma_1\times\gamma_2))$ is also invariant under $T_2$. We consider the quotient of $(\mathbb{C}^2, \omega_Q)$ under the action by $T_2$. Denote the resulting symplectic manifold by $(X,\omega)$, and let $\pi\colon \mathbb{C}^2\to X$ be the quotient map; this manifold is diffeomorphic to $\mathbb{R}\times S^1\times \mathbb{C}$, and the form $\omega$ is given by 

\begin{equation*}
    dx_1\wedge \frac{d\theta_1}{2\pi}+c(1-c)dx_2\wedge dy_2.
\end{equation*} 
The Hamiltonian $F$ reduces to a Hamiltonian on $X$; by an abuse of notation, we also denote its flow by $R_\theta$. The Lagrangian $\psi^{-1}_c(\gamma_1\times\gamma_2)$ projects to a Lagrangian infinite cylinder $L_\gamma$ in $(X, \omega)$. We denote $R_\theta(L_\gamma)$ by $L_\gamma^{\theta}$. The inclusion map $\iota\colon L_\gamma\to X$ induces an isomorphism on the fundamental group because $\pi^{-1}(L_\gamma)$ is $T_2$ invariant. The same holds for $L_\gamma^\theta$. Since our curves are periodic, $\mathcal{I}(\gamma_1, \gamma_2)$ is also invariant under the action by $T_2$, and its quotient is the set $L_\gamma\cap L_\gamma^\theta$. Hence, we need to prove this intersection is non-empty.

Our goal is to define Lagrangian Floer homology for the pair $(L_\gamma, L_\gamma^{\theta})$ and prove it is non-trivial. This will show the intersection $L_\gamma\cap L_\gamma^\theta$ is non-empty. Since the two Lagrangians are non-compact, we have to prove their intersection is compact in order to establish their Floer homology is well-defined and invariant under compactly supported Hamiltonian isotopies. 
The following lemma is the main result we will need for later compactness arguments.
\begin{lem}
\label{cpt lemma}
    Consider a positive number $N$ and $(z,w)\in\mathbb{C}^2$. Assume $z$, $z+cw$, and $z+ce^{\sqrt{-1}\theta}w$ are inside the strip $[-N,N]\times\mathbb{R}$. Then there is a constant $b(N,c,\theta)$ depending only on $N,c,$ and $\theta$ such that $\Vert w\Vert<b(N,c,\theta)$. 
\end{lem}
\begin{proof}
    Let $l$ in $[-N,N]\times\mathbb{R}$ be the line segment going through $z$ and parallel to the real line in $\mathbb{C}$. The point $z$ divides this segment into two parts; suppose $d_1$ is length of the section to the right of $z$ and $d_2$ is length of the left one. It is clear that both $d_1$ and $d_2$ are less than $2N$. If one of the points $z+cw$ or $z+ce^{\sqrt{-1}\theta}w$ lies on $l$, then we get
    \begin{equation*}
        c\Vert w\Vert<\max(d_1, d_2)\leq 2N.
    \end{equation*}
    Therefore, we can assume none of them is on $l$. Now suppose these points are on different sides of $l$, and let $\varphi_1$ denote the angle between $l$ and the line going through $z$ and $z+cw$. Define $\varphi_2$ for $z+ce^{\sqrt{-1}\theta}w$ in a similar way. We have $\varphi_1+\varphi_2=\theta$ and assume both $\varphi_1$ and $\varphi_2$ are not equal to $\frac{\pi}{2}$. We get
    \begin{align}
    \label{line eq1}
        &c\Vert w\Vert<\frac{\max(d_1, d_2)}{\cos(\varphi_1)}\leq\frac{2N}{\cos(\varphi_1)},\\
        \label{line eq2}
        & c\Vert w\Vert<\frac{\max(d_1, d_2)}{\cos(\varphi_2)}\leq\frac{2N}{\cos(\varphi_2)}
    \end{align}
    because we assumed both $z+cw$ and $z+ce^{\sqrt{-1}\theta}w$ are inside the strip. We conclude that
    \begin{equation*}
        \Vert w\Vert<\frac{2N}{c\cos(\frac{\theta}{2})}
    \end{equation*}
    since $\varphi_1+\varphi_2=\theta$. If one of the $\varphi_i's$ is equal to $\frac{\pi}{2}$, then the other one is $\theta-\frac{\pi}{2}$ and this proves
    \begin{equation*}
        \Vert w\Vert<\frac{2N}{c\sin(\theta)}.
    \end{equation*}
    Finally, suppose both of the points are on one side of $l$ and define the angles $\varphi_i$'s as in the previous case. In this case, we have $\varphi_1+\varphi_2=\pi-\theta$. We get the same inequalities as in \eqref{line eq1} and \eqref{line eq2}; hence,
    \begin{equation*}
        \Vert w\Vert<\frac{2N}{c\sin(\frac{\theta}{2})}.
    \end{equation*}
    This concludes the proof with 
    \begin{equation*}
        b(N, c, \theta)=\frac{2N}{c\cdot \min(\cos(\frac{\theta}{2}), \sin(\theta), \sin(\frac{\theta}{2}))}.
    \end{equation*}
\end{proof}

\begin{cor}
\label{comp Lemma}
    The intersection $L_\gamma\cap L_\gamma^{\theta}$ is compact.
\end{cor}
\begin{proof}
    We can find a positive number $N$ such that both $\gamma_1$ and $\gamma_2$ are inside $[-N, N]\times\mathbb{R}$. Consider $\pi(z,w)$ in $L_\gamma\cap L_\gamma^{\theta}$. Then $z$, $z+cw$, and $z+ce^{-\sqrt{-1}\theta}w$ are all inside this strip, so we can apply Lemma \ref{cpt lemma}. Thus $\Vert w \Vert$ is bounded, and this shows $\pi(z,w)$ is in $[-N, N]\times S^1\times D_{b(N,c,\theta)}$, where $D_{b(N,c,\theta)}$ is the closed disk with radius $b(N,c,\theta)$ around the origin in $\mathbb{C}$.  
\end{proof}

\section{Lagrangian Floer Homology}

For a brief discussion of Lagrangian-Floer homology, see \cite{11}. We will work with mod $2$ coefficients for singular homology and Floer homology in the following. In order to define Lagrangian Floer homology for the pair $(L_\gamma, L_\gamma^\theta)$, one needs to show holomorphic strips with boundary on these two Lagrangian have bounded image. Since the two Lagrangians are non-compact, this is only possible if we can control the behavior of $L_\gamma$ and $L_\gamma^\theta$ at infinity; see \cite{5, RittPaper}
for more details.
Apriori, this cannot be done for every pair of periodic curves $(\gamma_1, \gamma_2)$, however, we are able to change the pair $(L_\gamma, L_\gamma^\theta)$ to another pair $(\tilde{L}_\gamma, \tilde{L}_\gamma^\theta)$ so that 
\begin{equation*}
    L_\gamma\cap L_\gamma^\theta = \tilde{L}_\gamma\cap \tilde{L}_\gamma^\theta,
\end{equation*}
and we can define Floer homology for the pair $(\tilde{L}_\gamma, \tilde{L}_\gamma^\theta)$.
First we will explain the technical details on when Lagrangian Floer homology is well defined. Then we show how one can construct the pair $(\tilde{L}_\gamma, \tilde{L}_\gamma^\theta)$ from $(L_\gamma, L_\gamma^\theta)$ for a pair of smooth periodic curves $(\gamma_1, \gamma_2)$. 
Our ambient space $X$ is diffeomorphic to $\mathbb{R}\times S^1\times\mathbb{C}$, and we have the projection map $g\colon X\to \mathbb{R}$. We also have the standard complex structure $j$ on $X$ descending from the quotient map $\pi\colon\mathbb{C}^2\to X$.
\begin{defi}
\label{suitable defi}
Consider two Lagrangians $L_0$ and $L_1$ in $X$. We say these two Lagrangians are \textit{suitable} if they satisfy the following assumptions 
\begin{enumerate}
    \item The inclusion map $\iota\colon L_0\to X$ induces an isomorphism on $\pi_1(L_0)$.
    \item There exists a Hamiltonian diffeomorphism $\psi$ on $X$ such that $L_1 = \psi(L_0)$, and the intersection $L_0\cap L_1$ is compact.
    \item $g$ restricted to $L_0$ and $L_1$ is bounded.
    \item  There exists a constant $A$ such that if 
\begin{equation}
\label{the Hopf Ess}
    (x,v)\in (TL_0\cup TL_1)\cap \ker(dF)
\end{equation}
with $F(x)> A$, then $jv\in \ker(dF)$ where $F$ is the Hamiltonian defined in \eqref{ham defi}.
\end{enumerate}
\end{defi}
Let $(L_0, L_1)$ be a pair of Lagrangians in $(X, \omega)$. Define the path space $\mathcal{P}(L_0, L_1)$ by
\begin{equation*}
    \{x\in W^{1,2}([0,1], X)\hspace{1mm}|\hspace{1mm} x(0)\in L_0,\hspace{1mm} x(1)\in L_1\}.
\end{equation*}
Let 
\begin{equation*}
    \lambda = x_1\frac{d\theta_1}{2\pi}+c(1-c)x_2dy_2
\end{equation*}
be a primitive for $\omega$ on $X$, and consider a Hamiltonian $\widehat{H}\colon X\times[0,1]\to \mathbb{R}$. 
We define the symplectic action one form $\omega_{\widehat{H}}$ on the path space $\mathcal{P}(L_0, L_1)$ by
\begin{equation*}
    \omega_{\widehat{H}}(\xi) = \int_0^1 \omega(\xi, \Dot{x}(t)-X_{\widehat{H}}(x(t),t))dt
\end{equation*}
for every $x\in\mathcal{P}(L_0, L_1)$ and $\xi\in T_x\mathcal{P}(L_0, L_1)$. In general, this defines a closed one-form on the path space between two Lagrangians, but it is not necessarily exact; in our case, we have the following lemma. 
\begin{lem}
\label{Action Functional}
    Suppose $(L_0, L_1)$ is a suitable pair. Then
    the form $\omega_{\widehat{H}}$ on $\mathcal{P}(L_0, L_1)$ is exact for every Hamiltonian $\widehat{H}$.
\end{lem}
\begin{proof}
It suffices to show if $u\colon [0,1]\to \mathcal{P}(L_0, L_1)$ is a closed curve, then $\omega_{\widehat{H}}[u]=0$. We can view $u$ as a map from $S^1\times [0,1]$ to $X$, where $\beta_0\coloneqq u(, 0)$ is in $L_0$ and $\beta_1\coloneqq u(, 1)$ is in $L_1$. We get
\begin{equation*}
   \omega_{\widehat{H}}[u] =\int_{S^1}\int_0^1\omega(u_s, u_t-X_{\widehat{H}}(u,t))dtds =
   \int_u\omega-\int_0^1\int_{S^1}\frac{\partial\widehat{H}(u,t)}{\partial s}dsdt = \int_u\omega = \lambda(\beta_1)-\lambda(\beta_0), 
\end{equation*}
where in the last step, we applied Stokes' theorem. Let $\psi$ denote the Hamiltonian diffeomorphism that takes $L_0$ to $L_1$.
Firstly, note that $\beta_1$ and $\psi(\beta_0)$ are homotopic in $L_1$ because $\beta_0$ and $\psi(\beta_0)$ are homotopic in $X$, $u$ gives a homotopy between $\beta_0$ and $\beta_1$ in $X$, and the inclusion map from $L_1$ to $X$ induces an isomorphism on $\pi_1$. Secondly, we note that $\lambda$ is a closed one form on $L_1$, so we must have $\lambda(\psi(\beta_0))=\lambda(\beta_1)$.
Finally, since $\psi$ is a Hamiltonian diffeomorphism, the form $\psi^*(\lambda)-\lambda$ is exact; see \cite[Chapter 3]{14} for more details. Hence, we must have $\lambda(\beta_0)=\lambda(\psi(\beta_0))$.
\end{proof}

Consider a primitive for $\omega_{\widehat{H}}$ on $\mathcal{P}(L_0, L_1)$, and denote it by $\mathcal{A}_{\widehat{H}}$; we call this the symplectic action functional.  
Critical points of $\mathcal{A}_{\widehat{H}}$ are in one-to-one correspondence with the set $\phi_{\widehat{H}}^1(L_0)\cap L_1$ where $\phi_{\widehat{H}}^1$ denotes the time one Hamiltonian flow of $\widehat{H}$. Now consider a family of $\omega$-compatible time dependent almost complex structures $\{J_t\}_{0\leq t\leq 1}$ on $X$. This family induces a metric on $\mathcal{P}(L_0, L_1)$, and the gradient flow lines of $\mathcal{A}_{\widehat{H}}$ are in one-to-one correspondence with the solutions of

\begin{equation}
\label{CRF}
    u_s + J_t(u)(u_t-X_{\widehat{H}}(u))=0
\end{equation}
with boundary conditions 
\begin{equation}
    u(s,0) \in L_0 \hspace{1mm} u(s,1)\in L_1.
\end{equation}
This is called the perturbed Cauchy-Riemann equation (with respect to $\{J_t\}_{0\leq t\leq 1}$ and $\widehat{H}$) with boundary conditions on the pair $(L_0, L_1)$; we refer the unfamiliar reader to \cite[Section 1.3]{12}. For a solution $u$ to \eqref{CRF}, we define its energy by
\begin{equation*}
    E(u) = \int_{-\infty}^{\infty}\int_0^1 \Vert u_s\Vert^2 dtds
\end{equation*}
where $\Vert u_s\Vert^2$ is measured with respect to the Riemannian metric induced by $J_t$. If we have a solution with $E(u)<\infty$, then one can show there are $x,y\in \phi_{\widehat{H}}^1(L_0)\cap L_1$ with
\begin{equation*}
    \lim_{s\to -\infty} u(s,t) = x \hspace{1mm} \lim_{s\to +\infty} u(s,t) = y,
\end{equation*}
and 
\begin{equation}
\label{Action-Energy}
    E(u) = \mathcal{A}_{\widehat{H}}(y) - \mathcal{A}_{\widehat{H}}(x). 
\end{equation}

\begin{defi}
    We say a time-dependent Hamiltonian $\widehat{H}$ and a family of almost complex structures $\{J_t\}_{0\leq t\leq 1}$ on $X$ are admissible for a suitable pair $(L_0, L_1)$ if the following conditions hold.
    \begin{enumerate}
        \item $\phi_{\widehat{H}}^1(L_0)$ meets $L_1$ transversely, and $\widehat{H}$ is compactly supported.
        \item There exists a compact set $K\subset X$ such that $J_t$ agrees with $j$ on $X\setminus K$ for every $t$.
        \item The moduli spaces of finite energy solutions to \eqref{CRF} with boundary conditions on $(L_0, L_1)$ are smooth manifolds. 
    \end{enumerate}
\end{defi}

It is a standard fact in Floer theory that admissible pairs $(\widehat{H}, J_t)$ are abundant; we refer the unfamiliar reader to \cite{TrinFloerTheory} for more details.
Now consider a suitable pair of Lagrangians $(L_0,L_1)$ in $X$, and an admissible $(\widehat{H}, J_t)$ for this pair. Let $CF(L_0, L_1)$ be the vector space generated by the intersection points in $\phi_{\widehat{H}}^1(L_0)\cap L_1$ over $\mathbb{F}_2$. Define a morphism $\partial\colon CF(L_0, L_1)\to CF(L_0, L_1)$ by
\begin{equation*}
    \partial x = \sum_{y}\#\mathcal{M}_0(x,y)\cdot y
\end{equation*}
where $\mathcal{M}_0(x,y)$ is the zero dimensional part of the moduli space of solutions to \eqref{CRF} that goes from $x$ to $y$.

\begin{prop}
\label{partial sq=0}
     For a suitable pair of Lagrangians $(L_0,L_1)$ in $X$, we have $\partial^2 = 0$. 
\end{prop}
\begin{proof}
    Let $u$ be a holomorphic strip with boundary on $(L_0, L_1)$, and assume $E(u)$ is finite. Then we get
    \begin{equation*}
        E(u) = \mathcal{A}_{\widehat{H}}(u(\infty)) - \mathcal{A}_{\widehat{H}}(u(-\infty)) < C(L_0, L_1, \widehat{H})
    \end{equation*}
    where the constant $C$ does not depend on $u$ because we assumed $L_0\cap L_1$ is compact, so the critical values of $\mathcal{A}_{\widehat{H}}$ are bounded. Now we need to show $u$ stays inside a compact set that only depends on our choices. Define
    \begin{equation*}
        K = \{x\in X\colon \exists t \hspace{1mm} \text{with} \hspace{1mm} J_t(x)\neq j\}\cup \text{supp}(\widehat{H})\cup (\phi_{\widehat{H}}^1(L_0)\cap L_1).
    \end{equation*}
    This is a compact set, and if $u$ stays inside $K$, we get our claim. Therefore, suppose 
    $$\Omega\coloneqq u^{-1}(X\setminus K)$$ 
    is non-empty. Note that $g^2\circ u$ is a subharmonic function on $\Omega$, so we get
    \begin{equation*}
        \sup_{_\Omega}(g^2\circ u)\leq \sup_{\partial\Omega} g^2\circ u\leq \max(\sup_{L_0}g^2, \sup_{L_1}g^2, \sup_{K} g^2) = M(L_0, L_1, K).
    \end{equation*}
    Hence, we conclude that $g^2$ is bounded on every $u$ by $M$. The next step is to bound $F$ on $u$. Similarly, $F\circ u$ is a subharmonic function on $\Omega$, so it cannot have an interior maximum. Let $p\in \partial\Omega$ be a boundary maximum; we claim $u(p)\in K\cup F^{-1}[0, A]$ where $A$ is determined by the last assumption of Definition \ref{suitable defi}. By contradiction, suppose this is not the case. Then we must have $u(p)\in L_0$ or $L_1$. Without loss of generality, assume $u(p)\in L_0$. Since $p$ is a maximum point, we get
    \begin{equation*}
        u_s(p)\in TL_0\cap\ker(dF).
    \end{equation*}
    Thus $ju_s=u_t$ is also in $\ker(dF)$ by \eqref{the Hopf Ess}, but this is a contradiction because we must have
    \begin{equation*}
        \frac{\partial(F\circ u)}{\partial t}(p) = dF(u_t)\neq 0
    \end{equation*}
    by Hopf's lemma. We get that $F$ is bounded on $u$ by 
    \begin{equation*}
        D(L_0, L_1, K) = \max(A, \sup_K F). 
    \end{equation*}
    We conclude that every finite energy solution of \eqref{CRF} with boundary conditions on $(L_0, L_1)$ have energy less than $C$ and its image is contained in $[-\sqrt{M}, \sqrt{M}]\times S^1\times B_{\sqrt{D}}$, so we can apply Gromov's compactness theorem, and it only remains to show there are no sphere and disk bubbles. There is no sphere bubble because $\pi_2(X)$ is trivial. By our assumption on $(L_0, L_1)$, the inclusion map induces an isomorphism on the fundamental groups of $L_0$ and $L_1$; therefore, both $\pi_2(X,L_0)$ and $\pi_2(X,L_1)$ are trivial, so there is no disk bubble. We deduce that $\partial^2=0$. 
\end{proof}
Let $HF(L_0, L_1)$ denote the homology of $(CF(L_0, L_1), \partial)$.
\begin{prop}
\label{non-dependence}
    The homology $HF(L_0, L_1)$ does not depend on the choice of admissible $(\widehat{H}, J_t)$. 
\end{prop}
\begin{proof}
    Consider two admissible pairs $(\widehat{H}_0, J_t^0)$ and $(\widehat{H}_1, J_t^1)$. We can find a path of admissible pairs $(\widehat{H}_s, J_t^s)$ with $s\in \mathbb{R}$ such that 
    \begin{align*}
        &(\widehat{H}_s, J_t^s) = (\widehat{H}_0, J_t^0) \hspace{2mm} \forall s<<0, \\
        &(\widehat{H}_s, J_t^s) = (\widehat{H}_1, J_t^1) \hspace{2mm} \forall s>>0. 
    \end{align*}
    Now we can define a chain homotopy between $(CF_0(L_0, L_1), \partial_0)$ and $(CF_1(L_0, L_1), \partial_1)$ by counting the solutions of 
    \begin{equation*}
        u_s + J_t^s(u)(u_t-X_{\widehat{H}_s}(u)) = 0
    \end{equation*}
    with boundary conditions on $(L_0, L_1)$. Note that we can apply Gromov's compactness to these moduli spaces using the method of Proposition \ref{partial sq=0}.
\end{proof}
\begin{cor}
\label{compact supp invariance}
    Suppose $(L_0, L_1)$ is a suitable pair and $\phi\colon [0,1]\times X\to X$ is a compactly supported Hamiltonian isotopy. Then $(L_0, \phi_1(L_1))$ is also a suitable pair, and we have
    \begin{equation*}
        HF(L_0, L_1)\cong HF(L_0, \phi_1(L_1)).
    \end{equation*}
\end{cor}
\begin{proof}
  Since $\phi$ is compactly supported, it does not change the behavior of $L_1$ at infinity, and this shows the pair $(L_0, \phi_1(L_1))$ is suitable.  
  Consider a function $\beta\colon \mathbb{R}\to [0,1]$ so that $\beta(s) = 0$ for $s<<0$ and $\beta(s) = 1$ for $s>>0$. Choose an admissible pair $(\widehat{H}_0, J_t^0)$ for $(L_0, L_1)$ and an admissible pair $(\widehat{H}_1, J_t^1)$ for $(L_0, \phi(L_1))$. Interpolate between the two pairs as in the previous proposition; we count the finite energy solutions of 
  \begin{equation*}
      u_s + J_t^s(u)(u_t-X_{\widehat{H}_s}(u)) = 0
  \end{equation*}
  with boundary conditions
  \begin{equation*}
      u(s,0)\in L_0 \hspace{2mm} u(s,1)\in \phi_{\beta(s)}(L_1)\hspace{2mm} \forall s.
  \end{equation*}
  We can apply Gromov's theorem to the solutions of this equation using the same idea, and this gives us a chain homotopy between $CF(L_0, L_1)$ and $CF(L_0, \phi(L_1))$.
\end{proof}

Consider two disjoint vertical lines $(\delta_1, 
\delta_2)$ in $\mathbb{C}$. This is a pair of disjoint periodic curves, and we get a pair of Lagrangians $(L_\delta, L_\delta^\theta)$ in $X$ by applying the procedure in section $2$ to $(\delta_1, \delta_2)$.
\begin{lem}
    The pair $(L_\delta, L_\delta^\theta)$ is suitable.
\end{lem}
\begin{proof}
    The inclusion map $\iota\colon L_\delta\to X$ induces an isomorphism on the fundamental group by construction. We have $L_\delta^\theta = R_\theta(L_\delta)$, and the intersection $L_\delta\cap L_\delta^\theta$ is compact by corollary \ref{comp Lemma}. The pair $(\delta_1, \delta_2)$ is inside a strip $[-N, N]\times \mathbb{R}$, so $g$ is bounded on both Lagrangians. It remains to prove $(L_\delta, L_\delta^\theta)$ satisfy the last condition of definition \ref{suitable defi}. Let $d$ be the distance between $\delta_1$ and $\delta_2$ in $\mathbb{C}$. Take a point $(x,v)\in TL_\delta\cap dF$ with 
    \begin{equation}
    \label{more distance}
        F(x)>\frac{c(1-c)d^2}{2}.
    \end{equation} 
    Choose a point $(z,w)\in \pi^{-1}(x)$ and $\tilde{v}\in T_{(z,w)}\mathbb{C}^2$ with $d\pi(\tilde{v})=v$. Denote $\psi_c(z,w) = (p_1, p_2)\in \delta_1\times\delta_2$; we have $w = p_1-p_2$, and let $d\psi_c(\tilde{v}) = (v_1, v_2)\in T_{p_1}\delta_1\oplus T_{p_2}\delta_2$. We get
    \begin{equation*}
        w\cdot v_1-w\cdot v_2 = 0
    \end{equation*}
    because $v\in\ker(dF)$. Since $(\delta_1, \delta_2)$ are vertical lines, we have $(v_1, v_2) = (r\sqrt{-1}, s\sqrt{-1})$ for two real numbers $(r,s)$. Hence, we get
    \begin{equation*}
        (r-s)\cdot\text{Im}(w) = 0.
    \end{equation*}
    Now note $\text{Im}(w)\neq 0$ because $d(p_1, p_2)>d$ by equation \eqref{more distance}, so we must have $r=s$. Thus
    \begin{equation*}
        dF(jv)=w\cdot jv_1 - w\cdot jv_2= (s-r)\cdot\text{Re}(w) = 0.
    \end{equation*}
    Therefore, $L_\delta$ satisfies the last condition. Since $L_\delta^\theta = R_\theta(L_\delta)$, $dF$ is invariant under $R_\theta$, and $R_\theta$ is holomorphic, we conclude $L_\delta^\theta$ satisfies the last condition as well. 
\end{proof}

\begin{defi}
    We say two Lagrangians $L_0$ and $L_1$ have clean intersection if their intersection is a compact manifold $C$ in $X$, and we have
    \begin{equation*}
        T_xC = T_xL_0\cap T_xL_1
    \end{equation*}
    for every $x\in C$. This is a generalization of transverse intersection.
\end{defi} 
\begin{lem}
\label{Clean-Int-HF}
    Suppose $(L_0, L_1)$ is a suitable pair, and they intersect cleanly with $C=L_0\cap L_1$ a closed submanifold of $X$. Then
    \begin{equation*}
        HF(L_0, L_1)\cong H_*(C).
    \end{equation*}
\end{lem}
\begin{proof}
    It follows from \cite[Theorem 3.4.11]{6} that there is a spectral sequence from $H_*(C)$ to $HF(L_0, L_1)$. The higher differentials in this spectral sequence involve counting holomorphic strips with two ends on $C$, but there are no holomorphic strips with two ends on $C$ because the symplectic action one from is exact in our case; see \cite[Chapter 10]{ CleanFloerHomology} for more details. 
\end{proof}

\begin{prop}
\label{non-vanishing}
Let $\delta_1 = \{\alpha_1\}\times\mathbb{R}$ and $\delta_2 = \{\alpha_2\}\times\mathbb{R}$ be two distinct vertical lines in the plane. Then 
    $HF(L_\delta, L_\delta^\theta)\cong \mathbb{F}_2^2$.
\end{prop}
\begin{proof}
   Assume $\pi(z,w)$ is in $L_\delta\cap L_\delta^{\theta}$. Then we must have $z+cw\in \delta_1$ and $z+(c-1)w\in \delta_2$. Thus $z$ lies on the line $\{x=(1-c)\alpha_1+c\alpha_2\}$. We also have $z+ce^{-\sqrt{-1}\theta}w \in \delta_1$ and $z+(c-1)e^{-\sqrt{-1}\theta}w\in \delta_2$. Hence, we get
   \begin{equation*}
       w = (\alpha_1-\alpha_2)-(\alpha_1-\alpha_2)\tan(\frac{\theta}{2})\sqrt{-1}.
   \end{equation*}
   This shows the intersection of $\psi_c^{-1}(\delta_1\times\delta_2)$ and $R_\theta(\psi_c^{-1}(\delta_1\times\delta_2))$ is diffeomorphic to the real line, and can be identified with the image of the embedding
   \begin{equation*}
       s\in\mathbb{R}\mapsto ((1-c)\alpha_1+c\alpha_2+s\sqrt{-1}, (\alpha_1-\alpha_2)-(\alpha_1-\alpha_2)\tan(\frac{\theta}{2})\sqrt{-1})\in \mathbb{C}^2.
   \end{equation*}
    We claim this intersection is clean. Identify $\mathbb{C}^2$ with $\mathbb{R}^4$, and consider the coordinates $(x_1,y_1,x_2,y_2)$. At a point $(u_1,u_2)$ in $\psi_c^{-1}(\delta_1\times\delta_2)$, we have 
   \begin{equation*}
       T_{(u_1,u_2)}\psi_c^{-1}(\delta_1\times\delta_2) = \langle \partial_{y_1}, \partial_{y_2}\rangle,
   \end{equation*}
   and for a point $(u_1, u_2)$ in $R_\theta(\psi_c^{-1}(\delta_1\times\delta_2))$, we have
   \begin{equation*}
       T_{(u_1,u_2)}(R_\theta(\psi_c^{-1}(\delta_1\times\delta_2))) = \langle \partial_{y_1}, \cos(\theta)\partial_{y_2}-\sin(\theta)\partial_{x_2}\rangle.
   \end{equation*}
   We get 
   \begin{equation*}
       T_{(u_1,u_2)}\psi_c^{-1}(\delta_1\times\delta_2)\cap T_{(u_1,u_2)}(R_\theta(\psi_c^{-1}(\delta_1\times\delta_2))) = \langle \partial_{y_1}\rangle
   \end{equation*}
   because $\theta\in (0, \pi)$.
   This shows the intersection between $L_\delta$ and $L_\delta^\theta$ is also clean, since $\pi\colon \mathbb{C}^2\to X$ is a local diffeomorphism. Furthermore, $C\coloneqq L_\delta\cap L_\delta^\theta$ becomes a circle after applying $\pi$ to the line of intersection in $\mathbb{C}^2$. We conclude $HF(L_\delta, L_\delta^\theta)$ is isomorphic to $H_*(S^1)$ by Lemma \ref{Clean-Int-HF}.
\end{proof}

Consider a pair of smooth disjoint periodic curves $(\gamma_1, \gamma_2)$ in the plane. We are going to change $(L_\gamma, L_\gamma^\theta)$ to a suitable pair without changing their intersection. 
\begin{lem}
\label{vert to curve}
    There exists a pair of vertical lines $(\delta_1, \delta_2)$, and a time-dependent Hamiltonian $H$ on $X$ so that
    \begin{equation*}
        \phi_H^1(L_\delta) = L_\gamma,
    \end{equation*}
    and $\Vert H\Vert_{C^1(X)}<\infty$. Moreover, $H$ is supported in $[-N, N]\times S^1\times \mathbb{C}$ for a positive $N$.
\end{lem}
\begin{proof}
    Let $N$ be a positive number such that each $\gamma_i$ is inside $[-N, N]\times\mathbb{R}$. The quotient of $\mathbb{C}$ under translation by $\sqrt{-1}$ is an infinite cylinder, and we denote the quotient map by $q$. The pair $(\gamma_1,\gamma_2)$ is projected to a pair of disjoint simple closed curves in $[-N, N]\times S^1$ inside the infinite cylinder. We can find two distinct numbers $\alpha_1$ and $\alpha_2$ in $(-N,N)$ so that the signed area between $q(\gamma_i)$ and $\{\alpha_i\}\times S^1$ is zero for each $i$. Therefore, there are Hamiltonian functions $H_1$ and $H_2$ supported in $[-N,N]\times S^1$ such that the time one flow of $H_i$ takes $\{\alpha_i\}\times S^1$ to $q(\gamma_i)$ for each $i$. Define a Hamiltonian function $\tilde{H}$ on $(\mathbb{C}^2, \omega'_Q)$ by
\begin{equation*}
    \tilde{H}(z,w) = (1-c)H_1(q(z))+cH_2(q(w)).
\end{equation*}
This Hamiltonian is invariant under the action by $T_1$, and it takes $\delta_1\times\delta_2$ to $\gamma_1\times\gamma_2$ where $\delta_i=\{\alpha_i\}\times \mathbb{R}$. Therefore, this induces a Hamiltonian $H$ on $X$, and we have $\phi_H^1(L_\delta) = L_\gamma$. We have $\Vert H\Vert_{C^1(X)}<\infty$ because both $H_1$ and $H_2$ are compactly supported in the infinite cylinder. Note that $H$ is not necessarily compactly supported, but it is supported inside $[-N, N]\times S^1\times \mathbb{C}$ by construction. 
\end{proof}

\begin{prop}
\label{from curves to suitables}
    Assume $(\gamma_1, \gamma_2)$ is a pair of smooth disjoint periodic curves in the plane. Then there is a suitable pair $(\tilde{L}_\gamma, \tilde{L}_\gamma^\theta)$ in $X$ such that 
    \begin{equation*}
        \tilde{L}_\gamma\cap \tilde{L}_\gamma^\theta = L_\gamma\cap L_\gamma^\theta.
    \end{equation*}
    Furthermore, we have $HF(\tilde{L}_\gamma, \tilde{L}_\gamma^\theta)\cong \mathbb{F}_2^2$.
\end{prop}
\begin{proof}
    Choose a positive $N$ so that $\gamma_i$'s are inside $[-N, N]\times\mathbb{R}$. Apply Lemma \ref{vert to curve} to $(\gamma_1, \gamma_2)$ to find the vertical lines $(\delta_1, \delta_2)$ and the Hamiltonian $H$ with $L_\gamma = \phi_H^1(L_\delta)$. Denote the projection map $X\cong \mathbb{R}\times S^1\times\mathbb{C}\to \mathbb{C}$ by $p$. Let 
    \begin{equation*}
        S\coloneqq \sup_X \vert H\vert,\hspace{2mm}
        D\coloneqq \sup_X\Vert dp(X_H)\Vert.
    \end{equation*}
    All of these quantities are finite because $\Vert H\Vert_{C^1(X)}<\infty$. 
    Define a constant $B$ by
    \begin{equation*}
        B\coloneqq (D+S+2N)\cdot\sqrt{1+\big(\frac{1+\vert\cos\theta\vert}{\sin\theta}\big)^2}, 
    \end{equation*}
    and a constant $A$ by
    \begin{equation*}
        A\coloneqq \max(B, b(N,c,\theta))+D+S,
    \end{equation*}
    where the constant $b$ is determined by Lemma \ref{cpt lemma}.
    Consider a non-negative smooth function $h$ on $\mathbb{R}$ so that
    \begin{equation*}
        h(t)=
        \begin{cases}
             1 &t\leq A,\\
             0 &t>>0,
        \end{cases}
    \end{equation*}
    and $\vert h'(t)\vert\leq c(1-c)$ for every $t$. We define a Hamiltonian $G(x)\coloneqq h(\Vert p(x\Vert)\cdot H(x)$ on $X$; this function is compactly supported by our construction. Let
    \begin{equation*}
        \tilde{L}_\gamma \coloneqq \phi_G^1(L_\delta),\hspace{2mm} \tilde{L}_\gamma^\theta \coloneqq R_{\theta}(\tilde{L}_\gamma).
    \end{equation*}
    The pair $(\tilde{L}_\gamma, \tilde{L}_\gamma^\theta)$ is suitable because they are compactly supported Hamiltonian perturbations of the pair $(L_\delta, L_\delta^\theta)$, and we get
    \begin{equation*}
        HF(\tilde{L}_\gamma, \tilde{L}_\gamma^\theta)\cong HF(L_\delta, L_\delta^\theta)\cong\mathbb{F}_2^2
    \end{equation*}
    by Corollary \ref{compact supp invariance} and Proposition \ref{non-vanishing}.
    It remains to prove their intersection is the same as $(L_\gamma, L_\gamma^\theta)$. Denote the function $h(\Vert p(x)\Vert)$ by $T$. 
    First note that we have 
    \begin{equation*}
        X_T = \frac{h'(\Vert p(x)\Vert)\cdot \sqrt{-1}\cdot p(x)}{c(1-c)\Vert p(x)\Vert}.
    \end{equation*}
    Thus
    \begin{equation}
    \label{bounded norm}
        \Vert dp(X_T)\Vert = \frac{\vert h'(\Vert p(x)\Vert)\vert}{c(1-c)}\leq 1.
    \end{equation}
    Now we claim
    \begin{equation}
    \label{equality of intersections}
        \tilde{L}_\gamma\cap\{x\colon \Vert p(x)\Vert\leq \max(B, b(N, c, \theta))\}=L_\gamma\cap\{x\colon \Vert p(x)\Vert\leq \max(B, b(N,c,\theta))\}.
    \end{equation}
    Take a point $x\in \tilde{L}_\gamma$ with $\Vert p(x)\Vert\leq \max(B, b)$; we must have $x = \phi_G^1(y)$ with $y\in L_\delta$. We have
    \begin{equation}
    \label{bounding the diff}
    \begin{aligned}
        \Vert p(\phi_G^t(x)) - p(x)\Vert&= \Vert \int_0^t dp(X_G)dt\Vert\leq \Vert\int_0
        ^t T\cdot dp(X_H)+dp(X_T)\cdot Hdt\Vert
        \\
        &\leq\int_0^t\Vert dp(X_H)\Vert+\vert H\vert dt\leq D+S
    \end{aligned}
    \end{equation}
    for every $t\in[-1,1]$, where in the last step, we used equation \eqref{bounded norm}. Hence, we get 
    \begin{equation*}
        \Vert p(\phi_G^t(x))\Vert\leq \Vert p(x)\Vert+D+S\leq \max(B, b)+D+S = A
    \end{equation*}
    for every $t\in [-1,1]$. 
    In particular, this shows 
    \begin{equation*}
        x = \phi_G^1(y) = \phi_H^1(y)
    \end{equation*} 
    because by our construction, the two Hamiltonian functions $H, G$ agree on the points with $\Vert p(x)\Vert\leq A$.
    Therefore, $x$ is in $L_\gamma$ since $L_\gamma=\phi_H^1(L_\delta)$. Similarly, one can prove if $x\in L_\gamma$ and $\Vert p(x)\Vert\leq \max(B, b)$, then $x\in \tilde{L}_\gamma$. Now consider a point $x\in \tilde{L}_\gamma\cap\tilde{L}_\gamma^\theta$; we can find a $y$ in $L_\delta$ such that $x = \phi_G^1(y)$, and we have
    \begin{equation*}
        \vert Re(p(x)) - Re(p(y))\vert\leq \Vert p(x)-p(y)\Vert\leq D+S
    \end{equation*}
    by equation \eqref{bounding the diff}. Note that $\vert Re(p(y))\vert<2N$ since $y\in L_\delta$. Thus we get $\vert Re(p(x))\vert< D+S+2N$.
    Furthermore, we know $R_{-\theta}(x)\in \tilde{L}_\gamma$, and applying the same argument gives $\vert Re(e^{-i\theta}p(x))\vert<D+S+2N$. If we let $p(x) = u+iv$, then we get
    \begin{align*}
        & \vert u\vert < D+S+2N, \\
         \vert u\cos(\theta&)+v\sin(\theta)\vert<D+S+2N.
    \end{align*}
    Combining these two relations, we have
    \begin{equation}
    \label{bound after Perturbation}
        \Vert p(x)\Vert\leq (D+S+2N)\cdot\sqrt{1+\big(\frac{1+\vert\cos\theta\vert}{\sin\theta}\big)^2} = B.
    \end{equation}
     Therefore, if $x_1=R_\theta(x_2)\in \tilde{L}_\gamma\cap \tilde{L}_\gamma^\theta$, then $\Vert p(x_1)\Vert=\Vert p(x_2)\Vert\leq B$ by equation \eqref{bound after Perturbation}. Hence, $x_1$ and $x_2$ are in $L_\gamma$ by equation \eqref{equality of intersections}. On the other hand, if $x_1=R_\theta(x_2)$ lies in $L_\gamma\cap L_\gamma^\theta$, we must have $\Vert p(x_1)\Vert=\Vert p(x_2)\Vert\leq b$ by Lemma \ref{cpt lemma}. We get that both $x_i$'s are in $\tilde{L}_\gamma$ by equation \eqref{equality of intersections}. We conclude that $L_\gamma\cap L_\gamma^\theta= \tilde{L}_\gamma\cap\tilde{L}_\gamma^\theta$.
     
\end{proof}
\section{Proof of the Main Theorem}
\begin{proof}[\textbf{Proof of Theorem \ref{Main Thm}}]
We can find a positive number $N$ such that $\gamma_1$ and $\gamma_2$ are inside $[-N,N]\times\mathbb{R}$.
First assume $\gamma_1$ and $\gamma_2$ are smooth. By Proposition \ref{from curves to suitables}, we can construct a suitable pair of Lagrangians $(\tilde{L}_\gamma, \tilde{L}_\gamma^\theta)$ with 
$HF(\tilde{L}_\gamma, \tilde{L}_\gamma^\theta)\cong\mathbb{F}_2^2$ and $L_\gamma\cap L_\gamma^\theta=\tilde{L}_\gamma\cap\tilde{L}_\gamma^\theta$.

Hence, the intersection $L_\gamma\cap L_\gamma^\theta=\tilde{L}_\gamma\cap\tilde{L}_\gamma^\theta$ must be non-empty because $HF(\tilde{L}_\gamma, \tilde{L}_\gamma^\theta)$ is non-trivial; we conclude that $(\gamma_1, \gamma_2)$ admits a balanced inscription of $Q$. 
Moreover, $L_\gamma$ and $L_\gamma^\theta$ intersect transversely for a pair of generic smooth curves; thus in this case, we must have
\begin{equation*}
    \Bigg\vert\frac{\mathcal{I}(\gamma_1,\gamma_2)}{\text{Translation by $\sqrt{-1}$}}\Bigg\vert =|L_\gamma\cap L_\gamma^\theta|=\vert\tilde{L}_\gamma\cap \tilde{L}_\gamma^\theta\vert\geq \dim(HF(\tilde{L}_\gamma, \tilde{L}_\gamma^\theta))=2.
\end{equation*}
Therefore, for a pair of generic smooth curves, there are at least two balanced inscriptions of $Q$ that are not related under translation by $\sqrt{-1}$.

Now suppose $\gamma_1$ and $\gamma_2$ are continuous, and uniformly approximate each one by a sequence of smooth periodic curves $\gamma_i^n$ inside the strip $[-N,N]\times\mathbb{R}$ with $\gamma_1^n$ and $\gamma_2^n$ disjoint for every $n$. Let $Q_n$ denote an inscription of $Q$ inside the pair $(\gamma_1^n, \gamma_2^n)$. After translating each $Q_n$ by an integer multiple of $\sqrt{-1}$, we can assume the intersection point between the diagonals of $Q_n$ lies in $[-N,N]\times[0,1]$ for every $n$. Moreover, the diameter length of each $Q_n$ is bounded by a universal constant depending only on $N,\theta,$ and $c$ according to Lemma \ref{cpt lemma}. Thus the four vertices of each $Q_n$ lie inside a compact set in $\mathbb{C}$, and after passing to a subsequence, we can assume $Q_n$ converges to a trapezoid $\tilde{Q}$ with
\begin{equation*}
    \text{diam}(\tilde{Q})=\lim_{n\to \infty}\text{diam}(Q_n)\geq\lim_{n\to \infty}d(\gamma_1^n, \gamma_2^n)=d(\gamma_1, \gamma_2)>0.
\end{equation*}
The trapezoid $\tilde{Q}$ is similar to $Q$, and its four vertices lie on the pair $(\gamma_1, \gamma_2)$ in a balanced way.   
\end{proof}
\begin{proof}[\textbf{Proof of Corollary \ref{rect cor}}]
This follows from Theorem \ref{Main Thm} applied to the case $(c=\frac{1}{2}, \theta)$.
\end{proof}
\newpage

\bibliographystyle{amsplain} 
\bibliography{ref}

\begin{flushleft}
    Boston College. Massachusetts, USA.\\
    naserisa@bc.edu 
\end{flushleft}
\end{document}